\documentclass[a4paper,11pt]{amsart}
\usepackage[colorlinks, linkcolor=blue,anchorcolor=Periwinkle,
    citecolor=blue,urlcolor=Emerald]{hyperref}
\usepackage[all]{xy}
\SelectTips{cm}{}
\usepackage[mathscr]{euscript}

\newcommand{\ie}{{\em i.e.}\ }

\newcommand{\ko}{\: , \;}

\numberwithin{equation}{subsection}
\newtheorem{theorem}[subsection]{Theorem}
\newtheorem{classification-theorem}[subsection]{Classification Theorem}
\newtheorem{decomposition-theorem}[subsection]{Decomposition Theorem}
\newtheorem{proposition-definition}[subsection]{Proposition-Definition}
\newtheorem{periodicity-conjecture}[subsection]{Periodicity Conjecture}
\newtheorem{lemma}[subsection]{Lemma}

\newtheorem{conjecture}[subsection]{Conjecture}

\newtheorem{example}[subsection]{Example}

\newcommand{\reminder}[1]{}

\newcommand{\opname}[1]{\operatorname{\mathsf{#1}}}

\renewcommand{\mod}{\opname{mod}\nolimits}

\newcommand{\proj}{\opname{proj}\nolimits}
\newcommand{\Proj}{\opname{Proj}\nolimits}

\newcommand{\Mod}{\opname{Mod}\nolimits}

\newcommand{\per}{\opname{per}\nolimits}

\newcommand{\colim}{\opname{colim}}

\newcommand{\Z}{\mathbb{Z}}

\newcommand{\C}{\mathbb{C}}

\newcommand{\fm}{\mathfrak{m}}

\newcommand{\iso}{\xrightarrow{_\sim}}
\newcommand{\liso}{\xleftarrow{_\sim}}
\newcommand{\id}{\mathbf{1}}

\newcommand{\Hom}{\opname{Hom}}
\newcommand{\sg}{\opname{sg}}
\newcommand{\sgh}{\widehat{\opname{Sg}}}

\newcommand{\RHom}{\opname{RHom}}

\newcommand{\rad}{\opname{rad}}

\newcommand{\ten}{\otimes}
\newcommand{\lten}{\overset{\boldmath{L}}{\ten}}

\newcommand{\ca}{{\mathscr A}}
\newcommand{\cb}{{\mathscr B}}
\newcommand{\cc}{{\mathscr C}}
\newcommand{\cd}{{\mathscr D}}

\newcommand{\ch}{{\mathscr H}}

\newcommand{\cm}{{\mathscr M}}
\newcommand{\cn}{{\mathscr N}}

\newcommand{\cp}{{\mathscr P}}

\newcommand{\cs}{{\mathscr S}}

\renewcommand{\phi}{\varphi}
\newcommand{\del}{\partial}

\begin{document}

\date{June 3, 2018} 

\title[Singular Hochschild cohomology via the singularity category]{Singular Hochschild cohomology via the\\ singularity category}
\author{Bernhard Keller}
\address{Universit\'e de Paris\\
    Sorbonne Universit\'e\\
    UFR de Math\'ematiques\\
    CNRS\\
   Institut de Math\'ematiques de Jussieu--Paris Rive Gauche, IMJ-PRG \\   
    B\^{a}timent Sophie Germain\\
    75205 Paris Cedex 13\\
    France
}
\email{bernhard.keller@imj-prg.fr}
\urladdr{https://webusers.imj-prg.fr/~bernhard.keller/}



\begin{abstract}
We show that for a noetherian algebra $A$ whose bounded dg derived 
category is smooth, the singular Hochschild cohomology 
(=Tate--Hochschild cohomology)
is isomorphic, as a graded algebra, to the Hochschild cohomology
of the dg singularity category of $A$. The existence of such
an isomorphism is suggested by recent work of Zhengfang Wang.
\end{abstract}




\maketitle

\section{Introduction}

Let $k$ be a commutative ring. We write $\ten$ for $\ten_k$. Let $A$ be a
right noetherian (non commutative) $k$-algebra projective over $k$.
The {\em stable derived category} or {\em singularity category} of $A$ is
defined as the Verdier quotient 
\[
\sg(A)=\cd^b(\mod A)/\per(A)
\]
of the bounded derived category of finitely generated (right) $A$-modules by
the {\em perfect derived category $\per(A)$}, 
\ie the full subcategory of complexes quasi-isomorphic
to bounded complexes of finitely generated projective modules. It was introduced
by Buchweitz in an unpublished manuscript \cite{Buchweitz86} in 1986
and rediscovered, in its scheme-theoretic variant, by Orlov in 2003 \cite{Orlov04}.
Notice that it vanishes when $A$ is of finite global dimension and thus measures
the degree to which $A$ is `singular', a view confirmed by the results of
\cite{Orlov04}.

Let us suppose that the enveloping algebra $A^e=A\ten A^{op}$ is also right
noetherian. In analogy with Hochschild cohomology, in view of Buchweitz'
theory, it is natural to define the {\em Tate--Hochschild cohomology}
or {\em singular Hochschild cohomology} of $A$ to be the graded algebra with
components
\[
HH^n_{sg}(A,A) = \Hom_{\sg(A^e)}(A,\Sigma^n A)\ko n\in\Z \ko
\]
where $\Sigma$ denotes the suspension (=shift) functor.
It was studied for example in \cite{EuSchedler09,BerghJorgensen13,
Nguyen13} and more recently in \cite{Wang15, Wang15b, Wang16, RiveraWang17,
Wang18, ChenLiWang19}. 
Wang showed in \cite{Wang15} that, like Hochschild cohomology
\cite{Gerstenhaber63}, singular Hochschild cohomology carries a
structure of Gerstenhaber algebra. Now recall that the Gerstenhaber algebra
structure on Hochschild cohomology is a small part of much richer higher structure
on the Hochschild cochain complex $C(A,A)$ itself, namely the structure
of a $B_\infty$-algebra in the sense of Getzler--Jones \cite[5.2]{GetzlerJones94} 
(whose definition was motivated by \cite{Baues81})
given by the Hochschild differential, the cup product
and the brace operations \cite{Kadeishvili88a, Getzler93}. In \cite{Wang18},
Wang improves on \cite{Wang15} by defining a singular Hochschild
cochain complex $C_{sg}(A,A)$ and endowing it with a $B_\infty$-structure which
in particular yields the Gerstenhaber algebra structure on $HH^*_{sg}(A,A)$.

Using \cite{Keller03} Lowen--Van den Bergh showed in 
\cite[Theorem 4.4.1]{LowenVandenBergh05} that the Hochschild cohomology of $A$
is isomorphic to the Hochschild cohomology of the canonical differential
graded (=dg) enhancement of the
(bounded or unbounded) derived category of $A$ and that the isomorphism
lifts to the $B_\infty$-level (cf. Corollary 7.6 of \cite{Toen07} for a related
statement). Together with the
complete structural analogy between Hochschild and singular Hochschild
cohomology described above, this suggests the question whether
the singular Hochschild cohomology 
of $A$ is isomorphic to the Hochschild cohomology
of the canonical dg
enhancement $\sg_{dg}(A)$ of the singularity category $\sg(A)$
(note that such an enhancement exists by the construction of $\sg(A)$
as a Verdier quotient \cite{Keller99, Drinfeld04}). 
Our main result is the following.

\begin{theorem} \label{thm:main} Suppose that the commutative ground ring $k$
is of finite global dimension.
\begin{itemize}
\item[a)] There is a canonical morphism of graded algebras
$\Phi$ from the singular Hochschild cohomology of $A$ to the Hochschild
cohomology of the dg singularity category $\sg_{dg}(A)$.
\item[b)] If the bounded dg derived category $\cd^b_{dg}(\mod A)$ is
smooth, the morphism $\Phi$ is invertible.
\end{itemize}
\end{theorem}

Recall that a dg category $\ca$ is {\em smooth} if the identity bimodule
\[
I_\ca: (X,Y) \mapsto \ca(X,Y)
\]
is perfect in the derived category of $\ca$-bimodules (i.e. the derived
category of $\ca\lten_k\ca^{op}$). According to Theorem A of
Elagin--Lunts--Schn\"urer's \cite{ElaginLuntsSchnuerer18}, this
holds for $\cd^b_{dg}(\mod A)$
if $A$ is a finite-dimensional algebra over any field $k$ such that
$A/\rad(A)$ is separable over $k$ (which is automatic if $k$ is perfect).
By Theorem B of [loc.~cit.], it also holds if the algebra $A$ is 
right noetherian and finitely generated over its center and 
the center is a finitely generated algebra over $k$. 

Recall that localizations of smooth dg categories are smooth
(Proposition~3.10 c) of \cite{Keller11b}). So under the assumption of
b), the dg singularity category is also smooth. It seems an interesting
question to ask whether for arbitrary $A$, singular Hochschild
cohomology is Hochschild cohomology of {\em some} dg
category.

\begin{example} \label{ex:non separable} The author is grateful to
Xiao-wu Chen for pointing out this example: 
Suppose that $k$ is a field and $A\supset k$ a finite non separable
field extension. Then the singularity category is quasi-equivalent
to the zero dg category and so its Hochschild cohomology has
to vanish. Notice that $A$ and hence $A^e$ are selfinjective
so that the singularity category $\sg(A^e)$ is equivalent to
the stable module category. Since $A$ is a finite non separable
extension of $k$, the bimodule $A$ is non projective and hence
non zero in the stable module category. Therefore, its endomorphism
algebra, which is the zeroth singular Hochschild cohomology, 
is non zero. This shows that in general, the morphism $\Phi$ cannot be
invertible. Notice that in this case, the dg category $\cd^b_{dg}(\mod A)$
is derived Morita equivalent to the field $A$ and thus is not smooth over
the ground field $k$. 
\end{example}

\begin{conjecture} The morphism $\Phi$ of the theorem lifts to a morphism
\[
C_{sg}(A,A) \iso C(\sg_{dg}(A),\sg_{dg}(A))
\] 
in the homotopy category of $B_\infty$-algebras.
\end{conjecture}

Notice that the $B_\infty$-structure on Hochschild cohomology of dg categories
is preserved (up to quasi-isomorphism) under Morita equivalences, 
cf. \cite{Keller03}. 

\begin{theorem}[Chen--Li--Wang \cite{ChenLiWang19}] The conjecture
holds when $k$ is a field and $A=kQ/(kQ_1)^2$ is the radical square zero algebra
associated with a finite quiver $Q$ without sources or sinks with
arrow set $Q_1$.
\end{theorem}

Let us mention an application of Theorem \ref{thm:main} 
obtained in joint work with Zheng Hua. Suppose that $k$ is algebraically closed 
of characteristic $0$ and let $P$ the power series algebra $k[[x_1,\ldots, x_n]]$.

\begin{theorem}[\cite{HuaKeller18}] \label{thm:hypersurface} 
Suppose that $Q\in P$ has an isolated singularity at the origin and $A=P/(Q)$.
Then $A$ is determined up to
isomorphism by its dimension and the dg singularity category $\sg_{dg}(A)$.
\end{theorem}

In \cite[Theorem 8.1]{Efimov12}, Efimov proves a related but different
reconstruction theorem: He shows that if $Q$ is a polynomial, it is
determined, up to a formal change of variables, by the differential
$\Z/2$-graded endomorphism algebra $E$ of the residue field 
in the differential $\Z/2$-graded singularity category together
with a fixed isomorphism between $H^*B$ and the exterior
algebra $\Lambda(k^n)$.

In section \ref{s:main}, we generalize Theorem \ref{thm:main} to the
non noetherian setting and prove the generalized statement. We comment on
a possible lift of this proof to the $B_\infty$-level in section \ref{s:B-infinity}.
We prove Theorem \ref{thm:hypersurface}
in section \ref{s:hypersurface}.

\section{Generalization and proof}
\label{s:main}

\subsection{Generalization to the non noetherian case}
We assume that $A$ is an arbitrary $k$-algebra projective as a $k$-module. 
Its {\em pseudocoherent derived category} is the homotopy category of
right bounded complexes of finitely generated projective modules
with bounded homology:
\[
\cd^{pc}(A)=\ch^{-,b}(\proj A).
\]
Notice that when $A$ is right noetherian, this category is canonically equivalent to
$\cd^b(\mod A)$. It has a canoncial dg enrichment
\[
\cd^{pc}_{dg}(A) =\cc^{-,b}_{dg}(\proj A).
\]
The singularity category $\sg(A)$ is defined as the Verdier quotient 
$\cd^{pc}(A)/\ch^b(\proj A)$. When $A$ is right noetherian,
this is equivalent to the definition given in the introduction.

The (partially) {\em completed singularity category $\sgh(A)$} is defined as the Verdier
quotient of the  bounded derived category $\cd^b(\Mod A)$ of all right $A$-modules
by its full subcategory consisting of all complexes quasi-isomorphic to bounded complexes of arbitrary projective modules.

\begin{lemma} \label{lemma:fully-faithful}
The canonical functor $\sg(A) \to \sgh(A)$ is fully faithful.
\end{lemma}

\begin{proof} Let $M$ be a right bounded complex of finitely generated
projective modules with bounded homology and $P$ a bounded
complex of arbitrary projective modules. Since the components of
$M$ are finitely generated, each morphism $M \to P$ in the derived
category factors through a bounded complex $P'$ with finitely generated
projective components. This yields the claim.
\end{proof}

Since we do not assume that $A^e$ is noetherian, the $A$-bimodule $A$ will
not, in general, belong to the singularity category $\sg(A^e)$. 
But it always belongs to the completed singularity category $\sgh(A^e)$.
We define the singular Hochschild cohomology of $A$ to be the graded
algebra with components
\[
HH^n_{sg}(A,A) = \Hom_{\sgh(A^e)}(A, \Sigma^n A)\ko n\in\Z.
\]

\begin{theorem}  \label{thm:non-noetherian} Even if $A^e$ is non noetherian,
there is a canonical morphism of graded algebras
from the singular Hochschild cohomology of $A$ to the Hochschild
cohomology of the dg singularity category $\sg_{dg}(A)$. It is an
isomorphism if the pseudocoherent dg derived category $\cd^{pc}_{dg}(A)$
is smooth.
\end{theorem}

Let $P$ be a right bounded complex of projective $A^e$-modules.
For $q\in\Z$, let $\sigma_{>q} P$ and $\sigma_{\leq q} P$
denote its stupid truncations:
\begin{align*}
\sigma_{> q}P:
\xymatrix{\ldots \ar[r] & 0 \ar[r] & P^{q+1} \ar[r] & P^{q+1} \ar[r] & \ldots} \\
\sigma_{\leq q}P:
\xymatrix{\ldots \ar[r]  & P^{q-1} \ar[r] & P^q \ar[r] & 0 \ar[r] & \ldots}
\end{align*}
so that we have a triangle
\[
\xymatrix{
\sigma_{>q}P \ar[r] & P \ar[r] & \sigma_{\leq q}P \ar[r] & \Sigma \sigma_{>q}P.}
\]
We have a direct system
\[
\xymatrix{
P=\sigma_{\leq 0}P \ar[r] & \sigma_{\leq -1} P \ar[r] & \sigma_{\leq -2} P \ar[r] & \ldots \ar[r] & P_{\leq q} 
\ar[r] & \ldots}.
\]

\begin{lemma} \label{lemma:sgh} Let $L\in\cd^b(\Mod A^e)$.
We have a canonical isomorphism
\[
\colim \Hom_{\cd A^e}(L, \sigma_{\leq q}P) \iso
\Hom_{\sgh(A^e)}(L, P).
\]
In particular, if $P$ is a projective resolution of $A$ over $A^e$, we have
\[
\colim \Hom_{\cd A^e}(A, \Sigma^n \sigma_{\leq q}P) \iso
\Hom_{\sgh(A^e)}(A, \Sigma^n A)\ko n\in\Z.
\]
\end{lemma}

\begin{proof} Clearly, if $Q$ is a bounded complex of projective modules,
each morphism $Q \to P$ in the derived category $\cd A^e$ factors through
$\sigma_{>q}P \to P$ for some $q\ll 0$. This shows that the morphisms
$P \to \sigma_{\leq q}P$ form a cofinal subcategory in the category
of morphisms $ P \to P'$ whose cylinder is a bounded complex of
projective modules. Whence the claim.
\end{proof}

\subsection{Proof of Theorem~\ref{thm:non-noetherian}}
We refer to \cite{Keller94, Keller06d, Toen14} for foundational material
on dg categories. We will follow the terminology of \cite{Keller06d} and
use the model category structure on the category of dg categories
constructed in \cite{Tabuada05}.
For a dg category $\ca$, denote by $X \mapsto Y(X)$ the dg
Yoneda functor and by $\cd\ca$  the derived category. We write
$\ca^e$ for the enveloping dg category 
$\ca\lten_k\ca^{op}$ and $I_\ca$ for the {\em identity bimodule}
\[
I_\ca: (X,Y) \mapsto \ca(X,Y).
\]
By definition, the Hochschild cohomology of $\ca$ is the graded endomorphism
algebra of $I_\ca$ in the derived category $D(\ca^e)$.
In the case of the algebra $A$, the identity bimodule is the $A$-bimodule $A$.
Recall that if $F: \ca \to \cb$ is a fully faithful dg functor, the restriction
$F_*: \cd\cb \to \cd \ca$ is a localization functor admitting fully faithful left and
right adjoint functors $F^*$ and $F^!$ given respectively by
\[
F^*: M \mapsto M\lten_\ca \mbox{}_F\cb \quad\mbox{and}\quad
F^!: N \mapsto \RHom_\ca(\cb_F,N) \ko
\]
where $\mbox{}_F\cb=\cb(?,F-)$ and $\cb_F=\cb(F?,-)$. 

Let $\cm_0=\cc_{dg}^{-,b}(\proj A)$ denote the dg category of right bounded
complexes of finitely generated projective $A$-modules with bounded homology.
Notice that the morphism complexes of $\cm_0$ have terms which involve
infinite products of projective $A$-modules so that in general, the
morphism complexes of $\cm_0$ will not be cofibrant over $k$.
Let $\cm \to \cm_0$ be a cofibrant resolution of $\cm_0$. 
We assume, as we may, that the quasi-equivalence $\cm\to\cm_0$
is the identity on objects. Notice that the morphism complexes of $\cm$ are
cofibrant over $k$ so that we have $\cm\lten_k \cm^{op} \iso \cm\ten \cm^{op}$.
Let $\cp\subset \cm$ be the full dg subcategory of $\cm$ formed
by the bounded complexes of finitely generated projective $A$-modules.
Let $\cs$ denote 
the dg quotient  $\cm/\cp$. We assume, as we may, that $\cs$ is cofibrant. In the
homotopy category of dg categories, we have an isomorphism
between $\sg_{dg}(A)$ and $\cs=\cm/\cp$.
Let $B$ be the dg endomorphism algebra of $A$ considered as an object of
$\cp \subset \cm$. Notice that we have a quasi-isomorphism $B \to A$
and that both $B$ and $A$ are cofibrant over $k$.
We view $B$ as a dg category with one object whose endomorphism
algebra is $B$. We have the obvious inclusion and projection dg functors
\[
\xymatrix{
B \ar[r]^i & \cm \ar[r]^p & \cs.}
\]
Consider the fully faithful dg functors
\[
\xymatrix{
B\ten B^{op} \ar[r]^-{\id\ten i} & B\ten \cm^{op} \ar[r]^-{i\ten\id} & \cm\ten\cm^{op}}.
\]
The restriction along $G=\id\ten i$ admits the left adjoint $G^*$ given by
\[
G^*: X \mapsto \cm\mbox{}_i \lten_B X \ko
\]
and the restriction along $F=i\ten\id$ admits the fully faithful
left and right adjoints $F^*$ and $F^!$ given by
\[
F^*: Y \mapsto Y \lten_B \mbox{}_i\cm \quad\mbox{and} \quad
F^!: Y \mapsto \RHom_B(\cm_i, Y).
\]
Since $F^*$ and $F^!$ are the two adjoints of a localization functor,
we have a canonical morphism $F^* \to F^!$. 

\begin{lemma} If $P$ is an arbitrary sum of copies of $B^e$, the morphism
\[
F^* G^* (P) \to F^! G^* (P)
\]
is invertible. 
\end{lemma}

\begin{proof} Let $P$ be the direct sum of copies of $B^e$ indexed by a 
set $J$. Since $F^*$ and $G^*$ commute with (arbitrary) coproducts, 
the left hand side is the dg module
\[
\bigoplus_J \cm(i?,-)\lten_B (B\ten B) \lten_B \cm(?,i-) = 
\bigoplus_J \cm(B,-)\ten \cm(?,B),
\]
The right hand side is the dg module
\[
\RHom_B(\cm_i, \cm_i \lten_B (\bigoplus_J B\ten B)) = 
\RHom_B(\cm_i, \bigoplus_J \cm(B,-)\ten B).
\]
Let us evaluate the canonical morphism at $(M,L)\in \cm\ten\cm^{op}$. We find
the canonical morphism
\[
\bigoplus_J \cm(B,L) \ten \cm(M,B) \to \RHom_B(\cm(B,M),\bigoplus_j \cm(B,L)\ten B).
\]
We have quasi-isomorphisms
\[
\cm(B,L) \ten \cm(M,B) \to \cm_0(A,L) \ten \cm(M,B) \to L\ten \cm(M,B) \to L\ten \Hom_A(M,A)
\]
because $\cm(M,B)$ and $L$ are cofibrant over $k$. Now the equivalence
$\cd(B) \iso \cd(A)$ takes $\cm(B,L)\ten B$ to $\cm(B,L)\ten A \iso L\ten A$.
We have an quasi-isomorphism of dg $B$-modules $\cm(B,M) \iso \cm_0(A,M)=M$
and so the equivalence $\cd(B) \iso \cd(A)$ takes $\cm(B,M)$ to $M$.
Whence an isomorphism
\[
\RHom_B(\cm(B,M), \bigoplus_J \cm(B,L) \ten B) \iso
\RHom_A(M,\bigoplus_J L\ten A) = \Hom_A(M,\bigoplus_J L\ten A).
\]
Thus, we have to show that the canonical morphism
\[
\bigoplus_J L\ten \Hom_A(M,A) \to 
\Hom_A(M, \bigoplus_J L\ten A)
\]
is a quasi-isomorphism. Recall that $L$ and $M$ are right bounded complexes
of finitely generated projective modules with bounded homology. We fix $M$
and consider the morphism as a morphism of triangle functors with argument
$L\in\cd^b(\Mod A)$. This is possible thanks to the assumption that $k$ is of finite
global dimension. Now we are reduced to the case where $L$ is
in $\Mod A$. In this case, the morphism becomes an isomorphism of complexes
because the components of $M$ are finitely generated projective.
\end{proof}

Let us put $H=F^! G^* : \cd(B^e) \to \cd(\cm^e)$. Let us compute
the image of the identity bimodule $B$ under $H$. We have
\[
H(B) =F^! (\cm_i \lten_B B) = F^!(\cm_i) = \RHom_B(\cm_i, \cm_i)
\]
and when we evaluate at $L$, $M$ in $\cm$, we find
\[
H(B)(L,M)=\RHom_B(\cm(i ?,L),\cm(i ?,M)) = \RHom_B(\cm(B,L),\cm(B,M)).
\]
We have seen in the above proof that the equivalence $\cd(B) \iso \cd(A)$ takes
$\cm(B,L)$ to $L$. Whence quasi-isomorphisms
\begin{align*}
H(B)(L,M) = \RHom_B(\cm(B,L),\cm(B,M))  & \iso \RHom_A(L,M) = \Hom(L,M) \\
 & \liso \cm(L,M).
\end{align*}
Thus, the functor $H$ takes the identity bimodule $B$ to the identity
bimodule $I_\cm$. Since $F^!$ and $G^*$ are fully faithful
so is $H$. Denote by $\cn$ the image under the composition of $H$ 
with $\cd(A^e) \iso \cd(B^e)$ of the closure of $\Proj A^e$
under finite extensions. Then $H$ yields a fully faithful functor
\[
\sgh(A^e) \to \cd(\cm^e)/\cn
\]
taking the bimodule $A$ to the identity bimodule $I_\cm$.
Now notice that we have a Morita morphism of dg categories
\[
\cs^e \liso \frac{\cm\ten\cm^{op}}{\cp\ten\cm^{op}+\cm\ten\cp^{op}}.
\]
The functor $p^*: \cd(\cm^e) \to \cd(\cs^e)$ induces the quotient functor
\[
\xymatrix{
\frac{\cd(\cm\ten\cm^{op})}{\cn} \ar[r] &
\frac{\cd(\cm\ten\cm^{op})}{\cd(\cp\ten\cm^{op}+\cm\ten\cp^{op})}
=\cd(\cs^e)}.
\]
Since $p:\cm\to\cs$ is a localization, the image $p^*(I_\cm)$ is isomorphic
to $I_\cs$ (Proposition 3.10 a) of \cite{Keller11b}). 
This yields the morphism $\Phi$. Now let us assume
that $\cm$ is smooth. Since $H$ is fully faithful, to show that $\Phi$ is an 
isomorphism, it suffices to show that $p^*$ induces bijections
\[
\xymatrix{
\Hom_{\cd(\cm^e)/\cn}(I_\cm, \Sigma^n I_\cm) \ar[r] &
\Hom_{\cd(\cs^e)}(p^*(I_\cm), \Sigma^n p^*(I_\cm))}
\]
for each $n\in \Z$. Now by our smoothness assumption, the object
$I_\cm$ lies in $\per(\cm^e)$. The inclusion $\cp^e \to \cm^e$ induces
an equivalence of $\per(\cp^e)$ onto a full triangulated subcategory
of $\cn$.

\begin{lemma} The canonical functor
\[
\xymatrix{
\frac{\per(\cm^e)}{\per(\cp^e)} \ar[r] & \frac{\cd(\cm^e)}{\cn} }
\]
is fully faithful.
\end{lemma}
\begin{proof} This is a variation on Lemma~2.3 of \cite{Neeman92a}, cf. also
Lemma~5.3 of \cite{Keller94}. The claim follows from the following statement:
Let $X\in \per(\cm^e)$ and $Y\in \cd(\cm^e)$. Then each angle
\[
\xymatrix{
 & X \ar[d] \\
 Y \ar[r]_y & Y',}
\]
where the cone over $y$ is in $\cn$ may be completed to a commutative
square
\[
\xymatrix{
X' \ar[r]^x  \ar[d] & X \ar[d] \\
Y \ar[r]_y & Y'}
\]
where the cone over $x$ is in $\per(\cp^e)$. By construction, the cone
$N$ over $y$ is an $n$-fold iterated extension of (infinite) sums of
objects $P^\wedge$ for $P\in\cp^e$ for some $n\geq 1$. 
We proceed by induction on $n$. If $n=1$, then $N$ itself is a
sum of objects $P^\wedge$ and since $X$ is compact, the
composition $X \to Y' \to N$ factors through a finite subsum
$N'\subset N$. We can then form a morphism of triangles
\[
\xymatrix{
X' \ar[d] \ar[r]^x & X \ar[d] \ar[r] & N' \ar[d] \ar[r] & \Sigma X' \ar[d] \\
Y \ar[r]_y & Y' \ar[r] & N \ar[r] & \Sigma Y.}
\]
If $n>1$, then $N$ occurs in a triangle
\[
\xymatrix{
N_0 \ar[r] & N \ar[r] & N_1 \ar[r] & \Sigma N_0}
\]
for objects $N_0$ and $N_1$ which are $n'$-fold extensions
of sums of $P^\wedge$ for some $n'<n$. By forming an octahedron
over the morphisms $Y' \to N \to N_1$, we obtain a commutative diagram
\[
\xymatrix{
 & & N_0\ar[d] \ar@{=}[r] & N_0 \ar[d] \\
 Y \ar[d]_{y_2} \ar[r]^y & Y' \ar@{=}[d]  \ar[r] & N \ar[d] \ar[r] &
 					\Sigma Y \ar[d]^{\Sigma y_2} \\
 Y'' \ar[r]_{y_1} & Y' \ar[r] & N_1 \ar[d] \ar[r] & \Sigma Y'' \ar[d] \\
  & & \Sigma N_0 \ar@{=}[r] & \Sigma N_0
}
\]
We see that $y$ is the composition $y_1 \circ y_2$ of two morphisms
whose mapping cones are $n'$-fold extensions of sums of objects
$P^\wedge$, $P\in \cp^e$. By the induction hypothesis, we can find
a commutative diagram
\[
\xymatrix{
X'\ar[r]^{x_2} \ar[d] & X'' \ar[d] \ar[r]^{x_1} & X \ar[d] \\
Y \ar[r]_{y_2} & Y'' \ar[r]_{y_1} & Y',}
\]
where the cones over $x_1$ and $x_2$ lie in $\per(\cp^e)$. By the
octahedral axiom, this also holds for the cone over $x=x_1 \circ x_2$.
\end{proof}

Thus, by the Lemma, it suffices to show that $p^*$ induces bijections in the morphism
spaces with target $I_\cm$
\[
\xymatrix{
\Hom_{\per(\cm^e)/\per(\cp^e)}(?,I_\cm) \ar[r] &
\Hom_{\cd(\cs^e)}(p^*(?), p^*(I_\cm))}.
\]
Since the perfect derived category of $\cs^e$ is equivalent to the
quotient of $\per(\cm^e)$ by the thick subcategory generated by
the images under the Yoneda functor of objects in $\cp\ten\cm^{op}+\cm\ten\cp^{op}$,
it suffices to show that
$I_\cm$ is right orthogonal in $\per(\cm^e)/\per(\cp^e)$ on the images under
the Yoneda functor of the objects in $\cp\ten\cm^{op}+\cm\ten\cp^{op}$.
To show that $I_\cm$ is right orthogonal on $Y(\cm\ten\cp^{op})$, it
suffices to show that it is right orthogonal to an object $Y(M,B)$, $M\in\cm$.
Now a morphism in $\per(\cm^e)/\per(\cp^e)$ is given by a diagram 
of $\per(\cm^e)$ representing a left fraction
\[
\xymatrix{
Y(M,B) \ar[r] & I'_\cm & I_\cm \ar[l]}
\]
where the cone over $I_\cm \to I'_\cm$ lies in $\per(\cp^e)$. For each object $X$
of $\cd(\cm^e)$, we have canonical isomorphisms
\[
\Hom_{\cd\cm^e}(Y(M,B),X) = H^0(X(M,B)) =
\Hom_{\cd\cm}(Y(M), X(?,B)).
\]
Thus, the given fraction corresponds to a diagram in $\cd(\cm)$ of the form
\[
\xymatrix{
Y(M) \ar[r] & I'_\cm(?,B) & I_\cm(?,B)=\cm(?,B) \ar[l]} \ko
\]
where the cone over $I_\cm(?,B) \to I'_\cm(?,B)$ is the image under
$\cd A \iso \cd B\to \cd\cm$ of a bounded complex with projective components
(which may be infinitely generated because as a right module, the
tensor product $A\ten A$ may be infinitely generated).
Thus, the object $I'_\cm(?,B)$ is a direct factor
of a finite extension of shifts of arbitrary coproducts $B$.
Since $Y(M)$ is compact, the given morphism
$Y(M) \to I'_\cm(?,M)$ must then factor through $Y(Q)$ for an object $Q$ of $\cp$.
This means that the given morphism $Y(M,B) \to I'_\cm$ factors through
$Y(Q,B)$, which lies in $\cn$. Thus, the given fraction represents
the zero morphism of $\cd(\cm^e)/\cn$, as was to be shown.
The case of an object in $Y(\cp\ten\cm^{op})$ is analogous.
In summary, we have shown that the maps
\[
\xymatrix{
\sgh(A^e)(A, \Sigma^n A) \ar[r]^-H & (\cd(\cm^e)/\cn)(I_\cm,\Sigma^n I_\cm) 
\ar[r]^-{p^*} & \cd(\cs^e)(I_\cs, \Sigma^n I_\cs)}
\]
are bijective, which implies the assertion on Hochschild cohomology.

\section{Remark on a possible lift to the $B_\infty$-level}
\label{s:B-infinity}

Let $P \to A$ be a resolution of $A$ by 
projective $A$-$A$-bimodules. Let us assume for simplicity that $k$ is a field
so that we can take $\cm=\cm_0$ and $B=A$.
The proof in section~\ref{s:main}
produces in fact isomorphisms in the derived category of
$k$-modules
\begin{align*}
\colim \RHom_{A^e}(A, \sigma_{\leq q}P)  
& \to \colim \RHom_{\cm^e}(I_\cm, H \sigma_{\leq q}P) \\
& \to \colim \RHom_{\cs^e}(I_\cs, p^* H \sigma_{\leq q}P) \\
& = \RHom_{\cs^e}(I_\cs, I_\cs).
\end{align*}
For the bar resolution $P$, the truncation $\sigma_{\leq -q}P$
is canonically isomorphic to $\Sigma^q\Omega^q A$ so that the first
complex carries a canonical $B_\infty$-structure constructed by
Wang \cite{Wang18}.  As explained in the introduction,
it is classical that the last complex carries a canonical
$B_\infty$-structure. It is not obvious to make the intermediate
complexes explicit because the functor $H$, being a composition of
a {\em right adjoint} with a left adjoint to a restriction functor, does not take
cofibrant objects to cofibrant objects.

\section{Proof of Theorem~\ref{thm:hypersurface}}
\label{s:hypersurface}

By the Weierstrass preparation theorem, we may assume that $Q$ is
a polynomial. Let $P_0=k[x_1,\ldots, x_n]$ and $S=P_0/(Q)$. 
Then $S$ has isolated singularities but may have singularities other 
than the origin. Let $\fm$ be the maximal ideal of $P_0$
generated by the $x_i$ and let $R$ be the localization of 
$S$ at $\fm$. Now $R$ is local with
an isolated singularity at $\fm$ and $A$ is isomorphic to the
completion $\widehat{R}$. By Theorem~3.2.7 of \cite{Bach92}, in sufficiently 
high degrees $r$, the Hochschild cohomology of $S$ is isomorphic to 
the homology in degree $r$ of the complex
\[
k[u]\ten K(S, \del_1 Q, \ldots, \del_n Q) \ko
\]
where $u$ is of degree $2$ and $K$ denotes the Koszul complex.
Now $S$ is isomorphic to $K(P_0, Q)$
and so $K(S,\del_1 Q, \ldots,\del_n Q)$ is isomorphic to 
\[
K(P_0, Q, \del_1 Q, \ldots, \del_n Q). 
\]
Since $Q$ has isolated singularities,
the $\del_i Q$ form a regular sequence in $P_0$. So 
\[
K(P_0, Q, \del_1 Q, \ldots, \del_n Q)
\] 
is quasi-isomorphic to 
$K(M,Q)$, where $M=P_0/(\del_1 Q, \ldots, \del_n Q)$. Therefore,
in high even degrees $2r$, the Hochschild cohomology of $S$
is isomorphic to
\[
T=k[x_1, \ldots, x_n]/(Q,\del_1 Q, \ldots, \del_n Q)
\]
as an $S$-module. Since $S$ and $S^e$ are noetherian, this implies
that the Hochschild cohomology of $R$ in high even degrees is
isomorphic to the localisation $T_\fm$.
Since $R\ten R$ is noetherian and Gorenstein 
(cf. Theorem~1.6 of \cite{TousiYassemi03}), by Theorem~6.3.4 of \cite{Buchweitz86},
the singular Hochschild cohomology of $R$ coincides with 
Hochschild cohomology in sufficiently high degrees. 
Since $\C$ is a perfect field and $S$ a finitely generated $\C$-algebra,
the dg derived category $\cd^b_{dg}(\mod S)$ is smooth, by
Theorem~B of \cite{ElaginLuntsSchnuerer18}. 
The dg derived category $\cd^b_{dg}(\mod R)$ is a dg localization
of $\cd^b_{dg}(\mod S)$ and hence is also smooth. Thus, by
part b) of Theorem \ref{thm:main}, the Hochschild cohomology of $\sg_{dg}(R)$
is isomorphic to the singular Hochschild cohomology of $R$ and thus
isomorphic to $T_\fm$ in high even degrees. Since $R$ is a hypersurface,
the dg category $\sg_{dg}(R)$ is isomorphic, in the homotopy category
of dg categories, to the underlying differential $\Z$-graded category of
the differential  $\Z/2$-graded category of matrix factorizations of $Q$,
cf. \cite{Eisenbud80}, \cite{Orlov04} and Theorem~2.49 of
\cite{BlancRobaloToenVezzosi16}. 
Thus, it is $2$-periodic and so is its Hochschild cohomology. 
It follows that the zeroth Hochschild cohomology of $\sg_{dg}(R)$ is 
isomorphic to $T_\fm$ as an algebra. 
The completion functor $?\ten_R \widehat{R}$ yields an embedding
$\sg(R) \to \sg(A)$ through which $\sg(A)$ identifies
with the idempotent completion of the triangulated category $\sg(R)$,
cf. Theorem~5.7 of \cite{Dyckerhoff11}. Therefore, the corresponding
dg functor $\sg_{dg}(R) \to \sg_{dg}(A)$ induces an equivalence
in the derived categories and an isomorphism in Hochschild cohomology.
So we find an isomorphism
\[
HH^0(\sg_{dg}(A), \sg_{dg}(A)) \iso T_\fm.
\]
Since $Q \in k[x_1,\ldots, x_n]_\fm$ has an isolated singularity at the origin, 
we have an isomorphism
\[
T_\fm \iso k[[x_1, \ldots, x_n]]/(Q,\del_1 Q, \ldots, \del_n Q)
\]
with the Tyurina algebra of $A=P/(Q)$.
Now by the Mather--Yau theorem \cite{MatherYau82}, more
precisely by its formal version \cite[Prop. 2.1]{GreuelPham17}, in
a fixed dimension, the Tyurina algebra determines 
$A$ up to isomorphism.

Notice that the Hochschild cohomology of the dg category of matrix
factorizations considered as a differential $\Z/2$-graded category
is different: As shown by Dyckerhoff \cite{Dyckerhoff11}, it is isomorphic
to the Milnor algebra $P/(\del_1 Q, \ldots, \del_n Q)$ in
even degree and vanishes in odd degree. 

\section*{Acknowledgments}
I am very grateful to Zhengfang Wang for inspiring discussions on his results
and on the question which lead to this article. I am indebted to Zheng Hua 
and Xiaofa Chen for comments and to Xiao-Wu Chen for detecting embarrassing 
errors in several earlier versions of this note and in particular for
Example~\ref{ex:non separable}. I thank Greg Stevenson
for reference \cite{Efimov12}, Liran Shaul for reference \cite{TousiYassemi03}
and Amnon Yekutieli for pointing out a confusing misprint in the statement of
Lemma~\ref{lemma:sgh}.

\def\cprime{$'$} \def\cprime{$'$}
\providecommand{\bysame}{\leavevmode\hbox to3em{\hrulefill}\thinspace}
\providecommand{\MR}{\relax\ifhmode\unskip\space\fi MR }
\providecommand{\MRhref}[2]{%
  \href{http://www.ams.org/mathscinet-getitem?mr=#1}{#2}
}
\providecommand{\href}[2]{#2}


\begin{thebibliography}{10}

\bibitem{Baues81}
H.~J. Baues, \emph{The double bar and cobar constructions}, Compositio Math.
  \textbf{43} (1981), no.~3, 331--341.

\bibitem{BerghJorgensen13}
Petter~Andreas Bergh and David~A. Jorgensen, \emph{Tate-{H}ochschild homology
  and cohomology of {F}robenius algebras}, J. Noncommut. Geom. \textbf{7}
  (2013), no.~4, 907--937.

\bibitem{BlancRobaloToenVezzosi16}
Anthony Blanc, Marco Robalo, Bertrand To\"en, and Gabriele Vezzosi,
  \emph{Motivic realizations of singularity categories and vanishing cycles},
  arXiv:1607.03012 [math.AG].

\bibitem{Buchweitz86}
Ragnar-Olaf Buchweitz, \emph{Maximal {C}ohen--{M}acaulay modules and {T}ate
  cohomology over {G}orenstein rings}, \verb"http:hdl.handle.net/1807/16682"
  (1986), 155 pp.

\bibitem{ChenLiWang19}
Xiao-Wu Chen, Huanhuan Li, and Zhengfang Wang, \emph{The {H}ochschild
  cohomology of {L}eavitt path algebras and {T}ate--{H}ochschild cohomology},
  in preparation.

\bibitem{Drinfeld04}
Vladimir Drinfeld, \emph{D{G} quotients of {DG} categories}, J. Algebra
  \textbf{272} (2004), no.~2, 643--691.

\bibitem{Dyckerhoff11}
Tobias Dyckerhoff, \emph{Compact generators in categories of matrix
  factorizations}, Duke Math. J. \textbf{159} (2011), no.~2, 223--274.

\bibitem{Efimov12}
Alexander~I. Efimov, \emph{Homological mirror symmetry for curves of higher
  genus}, Adv. Math. \textbf{230} (2012), no.~2, 493--530.

\bibitem{Eisenbud80}
David Eisenbud, \emph{Homological algebra on a complete intersection, with an
  application to group representations}, Trans. Amer. Math. Soc. \textbf{260}
  (1980), no.~1, 35--64.

\bibitem{ElaginLuntsSchnuerer18}
Alexey Elagin, Valery~A. Lunts, and Olaf~M. Schn\"urer, \emph{Smoothness of
  derived categories of algebras}, arXiv:1810.07626 [math.AG].

\bibitem{EuSchedler09}
Ching-Hwa Eu and Travis Schedler, \emph{Calabi-{Y}au {F}robenius algebras}, J.
  Algebra \textbf{321} (2009), no.~3, 774--815.

\bibitem{Gerstenhaber63}
Murray Gerstenhaber, \emph{The cohomology structure of an associative ring},
  Ann. of Math. (2) \textbf{78} (1963), 267--288.

\bibitem{Getzler93}
Ezra Getzler, \emph{Cartan homotopy formulas and the {G}auss-{M}anin connection
  in cyclic homology}, Quantum deformations of algebras and their
  representations (Ramat-Gan, 1991/1992; Rehovot, 1991/1992), Israel Math.
  Conf. Proc., vol.~7, Bar-Ilan Univ., Ramat Gan, 1993, pp.~65--78.

\bibitem{GetzlerJones94}
Ezra Getzler and J.~D.~S. Jones, \emph{Operads, homotopy algebra, and iterated
  integrals for double loop spaces}, hep-th/9403055.

\bibitem{GreuelPham17}
Gert-Martin Greuel and Thuy~Huong Pham, \emph{Mather-{Y}au theorem in positive
  characteristic}, J. Algebraic Geom. \textbf{26} (2017), no.~2, 347--355.

\bibitem{Bach92}
Jorge~Alberto Guccione, Jose Guccione, Maria~Julia Redondo, and Orlando~Eugenio
  Villamayor, \emph{Hochschild and cyclic homology of hypersurfaces}, Adv.
  Math. \textbf{95} (1992), no.~1, 18--60.

\bibitem{HuaKeller18}
Zheng Hua and Bernhard Keller, \emph{Cluster categories and rational curves},
  arXiv:1810.00749 [math.AG].

\bibitem{Kadeishvili88a}
T.~V. Kadeishvili, \emph{{$A_\infty$}-algebra structure in cohomology and the
  rational homotopy type}, Preprint 37, Forschungsschwerpunkt Geometrie,
  Universit\"at Heidelberg, Mathematisches Institut, 1988.

\bibitem{Keller03}
Bernhard Keller, \emph{Derived invariance of higher structures on the
  {H}ochschild complex}, preprint, 2003, available at the author's home page.

\bibitem{Keller94}
\bysame, \emph{Deriving {D}{G} categories}, Ann. Sci. {\'E}cole Norm. Sup. (4)
  \textbf{27} (1994), no.~1, 63--102.

\bibitem{Keller99}
\bysame, \emph{On the cyclic homology of exact categories}, J. Pure Appl.
  Algebra \textbf{136} (1999), no.~1, 1--56.

\bibitem{Keller06d}
\bysame, \emph{On differential graded categories}, International Congress of
  Mathematicians. Vol. II, Eur. Math. Soc., Z\"urich, 2006, pp.~151--190.

\bibitem{Keller11b}
\bysame, \emph{Deformed {C}alabi--{Y}au completions}, Journal f{\"u}r die reine
  und angewandte Mathematik (Crelles Journal) \textbf{654} (2011), 125--180,
  with an appendix by Michel~Van den Bergh.

\bibitem{LowenVandenBergh05}
Wendy Lowen and Michel Van~den Bergh, \emph{Hochschild cohomology of abelian
  categories and ringed spaces}, Adv. Math. \textbf{198} (2005), no.~1,
  172--221.

\bibitem{MatherYau82}
John~N. Mather and Stephen S.~T. Yau, \emph{Classification of isolated
  hypersurface singularities by their moduli algebras}, Invent. Math.
  \textbf{69} (1982), no.~2, 243--251.

\bibitem{Neeman92a}
Amnon Neeman, \emph{The connection between the {K--theory} localisation theorem
  of {Thomason}, {Trobaugh} and {Yao}, and the smashing subcategories of
  {Bousfield} and {Ravenel}}, Ann. Sci. {\'E}cole Normale Sup{\'e}rieure
  \textbf{25} (1992), 547--566.

\bibitem{Nguyen13}
Van~C. Nguyen, \emph{Tate and {T}ate-{H}ochschild cohomology for finite
  dimensional {H}opf algebras}, J. Pure Appl. Algebra \textbf{217} (2013),
  no.~10, 1967--1979.

\bibitem{Orlov04}
D.~O. Orlov, \emph{Triangulated categories of singularities and {D}-branes in
  {L}andau-{G}inzburg models}, Tr. Mat. Inst. Steklova \textbf{246} (2004),
  no.~Algebr. Geom. Metody, Svyazi i Prilozh., 240--262.

\bibitem{Tabuada05}
Goncalo Tabuada, \emph{Une structure de cat\'egorie de mod\`eles de {Q}uillen
  sur la cat\'egorie des dg-cat\'egories}, C. R. Math. Acad. Sci. Paris
  \textbf{340} (2005), no.~1, 15--19. \MR{MR2112034 (2005h:18033)}

\bibitem{Toen07}
Bertrand To{\"e}n, \emph{The homotopy theory of {$dg$}-categories and derived
  {M}orita theory}, Invent. Math. \textbf{167} (2007), no.~3, 615--667.

\bibitem{Toen14}
Bertrand To\H{e}n, \emph{Lectures on dg-categories}, Topics in algebraic and
  topological {$K$}-theory, Lecture Notes in Math., vol. 2008, Springer,
  Berlin, 2011, pp.~243--302.

\bibitem{TousiYassemi03}
Masoud Tousi and Siamak Yassemi, \emph{Tensor products of some special rings},
  J. Algebra \textbf{268} (2003), no.~2, 672--676.

\bibitem{Wang18}
Zhengfang Wang, \emph{Gerstenhaber algebra and {D}eligne's conjecture on
  {T}ate--{H}ochschild cohomology}, arXiv:1801.07990.

\bibitem{Wang16}
\bysame, \emph{Singular deformation theory and the invariance of the
  {G}erstenhaber algebra structure on the singular {H}ochschild cohomology},
  arXiv:1501.01641 [math.RT].

\bibitem{Wang15}
\bysame, \emph{Singular {H}ochschild cohomology and {G}erstenhaber algebra
  structure}, arXiv:1508.00190 [math.RT].

\bibitem{Wang15b}
\bysame, \emph{Singular {H}ochschild cohomology of radical square zero
  algebras}, arXiv:1511.08348 [math.RT].

\bibitem{RiveraWang17}
Zhengfang Wang and Manuel Rivera, \emph{Singular {H}ochschild cohomology and
  algebraic string operations}, arXiv:1703.03899 [math.RT].

\end{thebibliography}

\end{document}